\documentclass{emsprocart}

\usepackage{amsfonts}
\usepackage{amscd,amsmath}
\usepackage{euscript}

\contact[romanvm@mi.ras.ru]{Roman Mikhailov, St Petersburg
Department of Steklov Mathematical Institute, and Chebyshev
Laboratory, St Petersburg State University, 14th Line, 29b, Saint
Petersburg 199178 Russia} \contact[ibspassi@yahoo.co.in]{Inder Bir
S. Passi, Centre for Advanced Study in Mathematics, Panjab
University, Sector 14, Chandigarh 160014 India, and Indian
Institute of Science Education and Research, Mohali (Punjab)
140306 India}

\numberwithin{equation}{section}

\newtheorem{theorem}{Theorem}[section]
\newtheorem{corollary}[theorem]{Corollary}

\theoremstyle{definition}

\usepackage[arrow,matrix,curve]{xy}\SilentMatrices
\def\xyma{\xymatrix@M.7em}

\newcommand{\ilimit}{\mbox{$\,\displaystyle{\lim_{\longleftarrow}}\,$}}

\title[Free Group Rings and Derived Functors]{Free Group Rings and Derived Functors}

\author[Roman Mikhailov and Inder Bir S. Passi]{Roman Mikhailov\thanks{The first author acknowledges Saint-Petersburg State University
for a research grant N 6.37.208.2016. This paper was completed
during the visit of the first author to Indian Institute of
Science Education and Reseach Mohali; he wishes to express his
gratitude to the Institute for its  warm hospitality.}\ and Inder
Bir S. Passi}

\begin{document}

\begin{abstract}
An approach to identify the  normal subgroups determined by ideals
in  free group rings with the help of the derived functors of
non-additive functors is explored. A similar approach, i.e., via
derived functors,  for  computing  limits of functors from the
category of free presentations to the category of abelian groups,
arising from commutator structure of free groups, is also
discussed.
\end{abstract}



\maketitle

\section{Introduction}
 Let $F$ be a free group and $\mathbb Z[F]$ its integral group ring. In the theory of group rings a  repeatedly occurring theme is
 the identification of normal subgroups $D(F,\,\mathfrak a):=F\cap (1+\mathfrak a)$ determined by two-sided ideals
 $\mathfrak a$ in $\mathbb Z[F]$ (see, e.g., \cite{Gupta:87}, \cite{MP:2009}, \cite{Vermani:1999}). It is often the case that a certain
 normal subgroup $N$, say, is easily seen to be contained in $D(F,\,\mathfrak a)$, and computing the quotient $D(F,\,\mathfrak a)/N$ usually
  becomes a challenging problem.  A classical instance of this phenomenon is the computation of the so-called dimension quotients
  $D(F,\, {\bf r}+{\bf f}^n)/R\gamma_n(F)$, $n\geq 1$, for a group $G$ with free presentation $G\cong F/R$, where $\bf r $ denotes the two-sided
  ideal of the group ring $\mathbb Z[F]$ generated by $R-1$, and $\gamma_n(F)$ is the $n$th term of the lower central series of $F$.
  More generally, if $R_1,\,R_2,\,\ldots,\, R_n$ are normal subgroups of $F$ and $\mathfrak a$ is a sum of certain products ${\bf r}_{i_1}\ldots {\bf r}_{i_s}$,
  then the identification of $D(F,\,\mathfrak a)$ is usually an intractable problem. However, it has recently been noticed that derived functors of non-additive
  functors in the sense of Dold-Puppe \cite{DP:1961} can be a useful tool for investigations in this area; for, the quotient $D(F,\,\mathfrak a)/N$ can
  sometime turn out to have interesting homological or homotopical interpretation.  To mention an instance, let $R$ be a normal subgroup
   of a free group $F$. It is well-known (see \cite{Gupta:87}) that $D(F,\,{\bf f}^n)=\gamma_n(F)$,
   for all $n\geq 1$, and $D(F, \, {\bf fr})=\gamma_2(R)$. Surprisingly, it turns out that
 $D(F, \, {\bf f}^3+{\bf fr})$ is related to the first derived functor of the second symmetric power functor: $L_1\operatorname{SP}^2(F/(\gamma_2(F)R))$
 \cite{HMP:2009}.  The purpose of the present study is to continue  further our work in \cite{MP:2015a} on the
 relationship between free group rings and derived functors of non-additive functors.
In another direction, we extend our results in \cite{MP:2015a} on
the connection between derived functors in the sense of
 Dold-Puppe and limits of functors from the category of free presentations of groups to the category of abelian groups.

 \par\vspace{.25cm}

Given a polynomial endofunctor $\mathcal F$ of degree $n$ (see
\cite{hpin}) on the category of abelian groups, say, for example,
the $n$th symmetric power ${\sf SP}^n$, the $n$th Lie power
$\mathfrak L^n$, the $n$th super-Lie power $\mathfrak L_s^n$, or
certain Schur functor, it has (in general, non-zero) derived
functors $L_i\mathcal F,\ i=0,\,1,\,2,\,\dots$, in the sense of
Dold-Puppe \cite{DP:1961}. It turns out that the zeroth and the
$(n-1)$st derived functors $L_0\mathcal F,\ L_{n-1}\mathcal F$
usually are the simplest ones to compute, and, in general they
have a simple combinatorial description. It is naturally to be
expected that they appear in our analysis. In general,
intermediate derived functors have a complicated structure, their
appearance in our study is indeed rather unexpected.
 \par\vspace{.25cm}
We begin by recalling, in Section 2, a needed basic fact about
free group rings. In Section 3 we develop the results on derived
functors of certain functors which help establish a connection
between derived functors and subgroups determined by ideals in
free group rings. The main results of this paper are in Section 4,
where we exhibit several quotients of subgroups determined by
ideals in free group rings in terms of derived functors. To
mention here just one of our  results, Theorem \ref{main} states
that if $1\to R\to F\to G\to 1$ is a free presentation of a group
$G$, and  $S$ is the commutator subgroup $[R,\,F]$, then there are
natural isomorphisms
\begin{align*}
& \frac{F\cap (1+{\bf rfr}+{\bf s}{\bf
r})}{\gamma_2(S)\gamma_3(R)}\cong L_1{\sf SP}^2(H_2(G)),\\
& \frac{F\cap (1+{\bf rfr}+{\bf
r}^2{\bf f})}{\gamma_2(S)\gamma_3(R)}\cong L_1{\sf SP}^2(H_2(G)),\\
& \frac{F\cap (1+{\bf s}^2{\bf r}+{\bf r}^2{\bf
fr})}{\gamma_3(S)\gamma_4(R)}\cong L_2{\mathfrak L}_s^3(H_2(G)),
\end{align*} where $H_2(G)$ is the second integral homology group of the group $G$. Finally, in Section 5, we give a number of identifications of the limits of functors, on the category of free presentations of groups, as derived functors. Again, to mention just one  result, we prove in Theorem \ref{main1} that if a group $G\cong F/R$ is 2-torsion-free, then
$$
\ilimit \frac{R''}{\gamma_2([\gamma_2(R),\,F])\gamma_3(R')}\cong
L_1{\sf SP}^2(H_4(G,\mathbb Z/2)),
$$where $R',\ R''$ are respectively the first and the second derived subgroups of $R$.
\par\vspace{.25cm}
For background on derived functors of non-additive functors, we
refer the reader to \cite{BM:2011} and  \cite{DP:1961}, and, for
free groups rings, to \cite{Gupta:87}.

\section{Preliminaries}\par\vspace{.5cm}
 For a normal subgroup $H$ of a group $G$, {\bf h}  denotes the two-sided ideal $(H-1)\mathbb Z[G]$  of the integral group ring
 $\mathbb Z[G]$. We denote by  $\gamma_n(G),\ n\geq 1,$  the $n$th term in the lower central series of $G$ defined inductively by setting
 \begin{multline*}
 G=\gamma_1(G), \quad \gamma_{n+1}(G)=[G,\,\gamma_n(G)]=\\ \langle [x,\,y]:=x^{-1}y^{-1}xy\,|\,x\in G,\ y\in \gamma_n(G)\rangle,\ n\geq 1.
 \end{multline*} For notational convenience, we also denote the derived subgroup $\gamma _2(G) $ by $G'$.
 \par\vspace{.25cm}
Let $F$ be a free group, $\mathfrak b\subset \mathfrak a$ and
$\mathfrak d\subset \mathfrak c$  ideals of ${\mathbb Z}[F]$ such
that $${\sf Tor}({\mathbb Z}[F]/\mathfrak a,\,{\mathbb
Z}[F]/\mathfrak c)=0,$$ where ${\sf Tor}={\sf Tor}^{\mathbb Z}_1.$
Then the map $(x,\,y)\mapsto xy$, $x\in \mathfrak a,\ y\in
\mathfrak c$,  induces an isomorphism (\cite{IM:2015}, Lemma 4.9.)
\begin{equation}\label{sergey}(\mathfrak a/\mathfrak b) \otimes_{\mathbb Z[F]} (\mathfrak
c/\mathfrak d)\cong \frac{\mathfrak a\mathfrak c}{\mathfrak
b\mathfrak c+\mathfrak a\mathfrak d}.
\end{equation}

 \section{Derived functors}

For a functor $T:\mathcal C\to \mathcal A$ from an abelian
catefory $\mathcal C$ to the category $\mathcal A$ of abelian
groups, $L_pT$ denotes the $p$th derived functor of $T$ at level
0, i.e., the functor $L_pT(-,\,0)$ in the notation of Dold-Puppe
\cite{DP:1961}. Recall that the functor $T: \mathcal C\to \mathcal
A$ is said to be a polynomial functor of degree $\leq n$ if the
$(n+1)$st cross-effect $T^{[n+1]}$ is zero \cite{hpin}.
\subsection{Quadratic functors}
 \par\vspace{.25cm}
Let $Q$ be a free abelian group, and $U$ a subgroup of $Q$. Then
we have the following commutative diagram with exact rows and
columns (see \cite{BM:2011},  \cite{Kock:2001}):
\begin{equation}\label{firstdia}
\xyma{& \Lambda^2(U)\ar@{>->}[r]\ar@{>->}[d] & U\otimes
Q\ar@{->}[r]\ar@{>->}[d] & {\sf
SP}^2(Q)\ar@{=}[d]\\
& \Lambda^2(Q)\ar@{>->}[r]\ar@{->>}[d] & Q\otimes Q\ar@{->}[r]\ar@{->>}[d] & {\sf SP}^2(Q)\\
L_1{\sf SP}^2(Q/U)\ar@{>->}[r] &
\Lambda^2(Q)/\Lambda^2(U)\ar@{->}[r] & Q/U\otimes Q\ar@{->>}[r] &
{\sf SP}^2(Q/U) }
\end{equation}where $\Lambda^2$ and ${\sf SP}^2$ are the exterior square and the symmetric square endofunctors respectively on the category $\mathcal A$, and the homomorphisms are the natural maps induced by the inclusion $U\subset Q$ or the projection $Q\to Q/U$.
Of particular interest to us is the lower 4-term exact sequence;
as such  we display it  separately for later
reference:\par\vspace{.25cm}
\begin{equation}\label{firstdialine}
0\to L_1{\sf SP}^2(Q/U)\to \Lambda^2(Q)/\Lambda^2(U)\to Q/U\otimes
Q\to {\sf SP}^2(Q/U)\to 0
\end{equation}
\par\vspace{.5cm}The following result, which is a generalization of the corresponding result in \cite{Kock:2001}, plays a crucial role in establishing a connection between subgroups determined by ideals in free group rings and derived functors. \par\vspace{.25cm}
\begin{theorem}\label{spespei}
Let $E$ be an abelian group, and $I$ a subgroup of $E$. The first
homology of the Koszul-type complex
$$
\Lambda^2(I)\to I\otimes E\to {\sf SP}^2(E),
$$
where the two homomorphisms are the natural maps induced by the
inclusion $I\subseteq E$, is naturally isomorphic to
$$
{\sf Coker}\{L_1{\sf SP}^2(E)\to L_1{\sf SP}^2(E/I)\}.
$$
\end{theorem}\par\vspace{.25cm}
\begin{proof}
Let  $Q$ be a free abelian group with  subgroups $U\subset V$ such
that $$ Q/U=E,\quad V/U=I.
$$
Consider the following diagram
$$
\xyma{& \Lambda^2(U)\ar@{>->}[r]\ar@{>->}[d] & U\otimes
Q\ar@{->}[r]\ar@{>->}[d] & {\sf
SP}^2(Q)\ar@{=}[d]\\
& \Lambda^2(V)\ar@{>->}[r]\ar@{->>}[d] & V\otimes Q\ar@{->}[r]\ar@{->>}[d] & {\sf SP}^2(Q)\\
K\ar@{>->}[r] & \Lambda^2(V)/\Lambda^2(U)\ar@{->}[r] & V/U\otimes
Q\ar@{->>}[r] & C}
$$
resulting from the natural map between representatives  \\
$L{\sf SP}^2(E): \xyma{& \Lambda^2(U)\ar@{>->}[r] & U\otimes
Q\ar@{->}[r] & {\sf SP}^2(Q)}
$\\
and \\
 $L{\sf SP}^2(E/I):
\xyma{
& \Lambda^2(V)\ar@{>->}[r] & V\otimes Q\ar@{->}[r]& {\sf SP}^2(Q)\\
}
$\\
of  $L_1{\sf SP}^2(E)$ and $L_1{\sf SP}^2(E/I)$ respectively in
the derived category of $\mathcal A$, where
$$
K={\sf Ker}\{L_1{\sf SP}^2(E)\to L_1{\sf SP}^2(E/I)\}
$$
and $C$ lives in the exact sequence
\begin{equation}\label{cexact}
L_1{\sf SP}^2(E)\to L_1{\sf SP}^2(E/I)\to C\to {\sf SP}^2(E) \to
{\sf SP}^2(E/I),
\end{equation}as can be seen by easy diagram chasing.
We assert that
\begin{equation}\label{ciso}
C=I\otimes E/{\sf Im}(\Lambda^2(I)\to I\otimes E). \end{equation}
To see this, let us present $Q$ as $F/F'$ with $F$  a free group
and $F'$ its derived subgroup, and let $R\subset S$ be normal
subgroups of $F$ such that
$$U=R/F', \quad V=S/F'.$$ With our notation for ideals in group rings induced by normal subgroups, we have, in view of (\ref{sergey}),  natural isomorphisms
$$
I\otimes Q\cong\frac{{\bf s}+{\bf f}^2}{{\bf r}+{\bf f}^2}\otimes
\frac{\bf f}{{\bf f}^2}\cong \frac{{\bf sf}+{\bf f}^3}{{\bf
rf}+{\bf f}^3}.
$$
Therefore, there is a natural isomorphism
$$
C\cong\frac{{\bf sf}+{\bf f}^3}{(S'-1)+{\bf rf}+{\bf f}^3}
$$
On the other hand,
$$
I\otimes E\cong \frac{{\bf s}+{\bf f}^2}{{\bf r}+{\bf f}^2}\otimes
\frac{\bf f}{{\bf r}+{\bf f}^2}\cong \frac{{\bf sf}+{\bf
f}^3}{{\bf rf}+{\bf sr}+{\bf f}^3}
$$
Now observe that
$$
I\otimes E/{\sf Im}(\Lambda^2(I)\to I\otimes E)\cong \frac{{\bf
sf}+{\bf f}^3}{(S'-1)+{\bf rf}+{\bf sr}+{\bf f}^3}=\frac{{\bf
sf}+{\bf f}^3}{(S'-1)+{\bf rf}+{\bf\, f}^3}\cong C,
$$
and thus the isomorphism (\ref{ciso}) is proved. Consequently the
assertion in the Theorem follows from the exact sequence
(\ref{cexact}).
\end{proof}

\par\vspace{.5cm}Given a subgroup $I$ of an abelian group $E$, let $\underline{\Lambda^2(I)}$ denote the image of the map $\Lambda^2(I)\to \Lambda^2(E)$
induced by the natural inclusion map $I\hookrightarrow E$.

\par\vspace{.5cm}

\begin{theorem}\label{lambdalemma} If  $E$ is  an abelian group, $I$ a subgroup of $E$ and  $${\sf Tor}(E/I,\,E)\to L_1{\sf SP}^2(E/I)$$
the composition of the two natural maps $${\sf Tor}(E/I,\,E)\to
{\sf Tor}(E/I,\,E/I)\to L_1{\sf SP}^2(E/I),$$ then there is a
natural isomorphism
$$
{\sf Coker}\{{\sf Tor}(E/I,\,E)\to L_1{\sf SP}^2(E/I)\}\cong {\sf
Ker}\{\Lambda^2(E)/\underline{\Lambda^2(I)}\to E/I\otimes E\}.
$$

\end{theorem}\par\vspace{.25cm}
\begin{proof}
Consider the following commutative diagram with exact columns:
$$
\xyma{& {\sf Tor}(E/I,\,E)\ar@{->}[d]\\ \Lambda^2(I) \ar@{->}[r]
\ar@{->}[d]& I\otimes E\ar@{->}[r]\ar@{->}[d] & {\sf
SP}^2(E)\ar@{=}[d]\\
\Lambda^2(E)\ar@{>->}[r]\ar@{->>}[d] & E\otimes E\ar@{->>}[r]\ar@{->>}[d] & {\sf SP}^2(E)\\
\Lambda^2(E)/\underline{\Lambda^2(I)} \ar@{->}[r] & E/I\otimes E}
$$
Note that the middle horizontal sequence is exact. The homology
exact sequence together with Theorem \ref{spespei} implies that
there is a natural exact sequence
$$
{\sf Tor}(E/I,\,E)\to {\sf Coker}\{L_1{\sf SP}^2(E)\to L_1{\sf
SP}^2(E/I)\}\twoheadrightarrow{\sf
Ker}\{\Lambda^2(E)/\underline{\Lambda^2(I)}\to E/I\otimes E\}.
$$
However,
$$
{\sf Im}\{L_1{\sf SP}^2(E)\to L_1{\sf SP}^2(E/I)\}\subseteq {\sf
Im}\{{\sf Tor}(E/I,\,E)\to L_1{\sf SP}^2(E/I)\}
$$
and we thus obtain the asserted isomorphism.
\end{proof}
\par\vspace{.25cm}
\subsection{Cubic functors}
For an abelian group $A$, recall that  $\mathfrak L_s^3(A)$,  the
{\it third super Lie functor}  evaluated at $A $ \cite{BM:2011},
is by definition the abelian  group generated by brackets
$\{a,\,b,\,c\}$, $a,\,b,\,c\in A$, which are additive in each
variable, and satisfy the following defining relations:
\begin{align*}
& \{a,\,b,\,c\}=\{b,\,a,\,c\},\\
& \{a,\,b,\,c\}+\{c,\,a,\,b\}+\{b,\,c,\,a\}=0.
\end{align*}

Let $Q$ be a free abelian group, and $U$ a subgroup of $Q$. We
note that the following diagram
$$
\xyma{\EuScript L^3(U)\ar@{>->}[d] \ar@{>->}[r] & U\otimes
U\otimes Q\ar@{->}[r] \ar@{>->}[d] & U\otimes Q\otimes
Q\ar@{->}[r]\ar@{>->}[d] & \EuScript L_s^3(Q)\ar@{=}[d]\\
\EuScript L^3(Q) \ar@{->>}[d] \ar@{>->}[r] & Q\otimes Q\otimes
Q\ar@{->>}[d] \ar@{->}[r] &
Q\otimes Q\otimes Q\ar@{->>}[r] \ar@{->>}[d] & \EuScript L_s^3(Q)\\
\EuScript L^3(Q)/\EuScript L^3(U) \ar@{->}[r] & \frac{Q\otimes
Q}{U\otimes U}\otimes Q\ar@{->}[r] & Q/U\otimes Q\otimes Q}
$$
yields the natural exact sequence
\begin{equation}\label{l3seq}
0\to L_2{\mathfrak L}_s^3(Q/U)\to \EuScript L^3(Q)/\EuScript
L^3(U)\to \frac{Q\otimes Q}{U\otimes U}\otimes Q.
\end{equation}

\subsection{Metabelian Lie functor} Let $1\to R\to F\to G\to 1$ be a free presentation of a group $G$. Let $H_{ab}$ denote the abelianization $H/H'$ of the group $H$.
Again, it may be noted that the following diagram
$$
\xyma{\Lambda^2(\bar R)\otimes {\sf SP}^{n-1}(F_{ab}) \ar@{->}[d]
\ar@{->}[r] & \bar R\otimes {\sf SP}^n(F_{ab}) \ar@{>->}[d]
\ar@{->}[r] & {\sf
SP}^{n+1}(F_{ab})\ar@{=}[d]\\
\frac{\gamma_{n+1}(F)}{(\gamma_{n+1}(F)\cap
F'')\gamma_{n+2}(F)}\ar@{->>}[d] \ar@{>->}[r] & F_{ab}\otimes {\sf
SP}^n(F_{ab})\ar@{->>}[d] \ar@{->>}[r] & {\sf SP}^{n+1}(F_{ab})\\
\frac{\gamma_{n+1}(F)}{[R,\,R,\,F,\,F,\,\ldots\,,\,F_{n-1\
terms}\,](F''\cap \gamma_{n+1}(F))\gamma_{n+2}(F)}\ar@{->}[r] &
G_{ab}\otimes {\sf SP}^n(F_{ab}) }
$$
yields the natural exact sequence
\begin{multline}\label{spn}
0\to L_1{\sf SP}^{n+1}(G_{ab})\to
\frac{\gamma_{n+1}(F)}{[R,\,R,\,\underbrace{F,\,F,\,\ldots,\,F}_{n-1}\,](F''\cap
\gamma_{n+1}(F))\gamma_{n+2}(F)}\to\\ G_{ab}\otimes {\sf
SP}^n(F_{ab})\to {\sf SP}^{n+1}(G_{ab})\to 0,
\end{multline}
where $F''$ is the second derived subgroup of $F$.
 \par\vspace{.5cm}

 \section{Identification theorems}
\begin{theorem} Let $R$ and $S$ be normal subgroups of a free group $F$. Then there is a  natural isomorphism
\begin{align}
& \frac{F\cap (1+{\bf f}^2{\bf r}+{\bf sr})} {\gamma_2(R\cap
(F'S))\gamma_3(R)}\cong L_1{\sf SP}^2\left(\frac{R}{R\cap
(F'S)}\right)\label{rffrs}.
\end{align}
Moreover, if $S$ is a normal subgroup of $R$, then
\begin{align}
& \frac{F\cap (1+{\bf rfr}+{\bf sr})}{\gamma_3(R)S'}\cong L_1{\sf
SP}^2\left(\frac{R}{SR'}\right)\label{rfrrs},\\
& \frac{F\cap (1+{\bf r}^2{\bf fr}+{\bf s}^2{\bf
r})}{\gamma_4(R)\gamma_3(S)}\cong L_2{\mathfrak
L}_s^3\left(\frac{R}{SR'}\right).\label{rrfr}
\end{align}
\end{theorem}
\vspace{.5cm}\noindent{\it Proof of (\ref{rffrs}).} Let us set
$$
Q:=R/R'=R_{ab}\cong \frac{\bf r}{\bf fr},\quad U:=\frac{R\cap
(SF')}{R'}.
$$
Observe that there is a natural monomorphism
$$
\frac{R}{R\cap (SF')}\otimes R_{ab}\hookrightarrow \frac{\bf
f}{{\bf s}+{\bf f}^2}\otimes \frac{\bf r}{\bf fr}\cong \frac{\bf
fr}{{\bf sr}+{\bf f}^2{\bf r}}
$$
resulting  from the well-known identification of the second
dimension subgroup: $$F\cap (1+{\bf s}+{\bf f}^2)=SF'.$$ Next
observe that
$$
\Lambda^2(Q)/\Lambda^2(U)\cong \frac{R'} {\gamma_2(R\cap
(SF'))\gamma_3(F)},
$$and
$$
F\cap (1+{\bf sr}+{\bf f}^2{\bf r})\subseteq R',
$$
since ${\bf sr}+{\bf f}^2{\bf r}\subset {\bf fr}$ and $F\cap
(1+{\bf fr})=R'$. Thus (\ref{rffrs}) follows from
(\ref{firstdialine}).\ \ $\Box$

\par\vspace{.5cm}\noindent{\it Proof of (\ref{rfrrs}).} Let us set
$$
Q:=R/R'=R_{ab}\cong\frac{\bf r}{\bf fr},\quad U:=\frac{SR'}{R'}.
$$
Then \begin{align*}
& \Lambda^2(Q)/\Lambda^2(U)\cong \frac{R'}{S'\gamma_3(R)}\\
& Q/U\otimes Q\cong \frac{\bf r}{\bf s+rf}\otimes \frac{\bf r}{\bf
fr}\cong \frac{{\bf r}^2}{{\bf sr}+{\bf rfr}}.
\end{align*}
Since $F\cap (1+{\bf sr}+{\bf rfr})\subseteq R'$, the isomorphism
(\ref{rfrrs})  follows from (\ref{firstdialine}).\ \ $\Box$

\par\vspace{.5cm}\noindent {\it Proof of
(\ref{rrfr}).} In order to prove (\ref{rrfr}), observe that
\begin{align*}
& \EuScript L^3(Q)/\EuScript L^3(U)\cong\frac{\gamma_3(R)}{\gamma_3(S)\gamma_4(R)}\\
& \frac{Q\otimes Q}{U\otimes U}\otimes Q\cong\frac{\bf r^2}{{\bf
s}^2+{\bf r}^2{\bf f}}\otimes \frac{\bf r}{\bf fr}\cong\frac{{\bf
r}^3}{{\bf s}^2{\bf r}+{\bf r}^2{\bf fr}}.
\end{align*}
Since $F\cap (1+ {\bf s}^2{\bf r}+{\bf r}^2{\bf fr})\subseteq
F\cap (1+{\bf r}^2{\bf f})=\gamma_3(R)$, the isomorphism
(\ref{rrfr}) follows from the exact sequence (\ref{l3seq}).\ \
$\Box$

\bigskip

We next exhibit certain quotients constructed from a free
presentation $$1\to R\to F\to G\to 1$$ of a group $G$ which are
independent of the chosen free presentation, and in fact depend
only on the second integral homology group $H_2(G)$.
\par\vspace{.25cm}
\begin{theorem}\label{main} Let $1\to R\to F\to G\to 1$ be a free presentation of a group $G$, and let $S= [R,\,F]$. Then there are natural isomorphisms
\begin{align}
& \frac{F\cap (1+{\bf rfr}+{\bf s}{\bf
r})}{\gamma_2(S)\gamma_3(R)}\cong L_1{\sf SP}^2(H_2(G))\label{sph21}.\\
& \frac{F\cap (1+{\bf rfr}+{\bf
r}^2{\bf f})}{\gamma_2(S)\gamma_3(R)}\cong L_1{\sf SP}^2(H_2(G))\label{sph22}.\\
& \frac{F\cap (1+{\bf s}^2{\bf r}+{\bf r}^2{\bf
fr})}{\gamma_3(S)\gamma_4(R)}\cong L_2{\mathfrak
L}_s^3(H_2(G)).\label{sph23}
\end{align}
\end{theorem}

\vspace{.5cm}\noindent{\it Proof of (\ref{sph21}) and
(\ref{sph22}).} Let us set $$Q:=R/R'=R_{ab}\cong\frac{\bf r}{\bf
fr},\quad U:=S/R'.
$$
Then
$$
Q/U\otimes Q\cong R/S\otimes \frac{\bf r}{\bf fr}\cong\frac{\bf
r}{(S-1)+{\bf rf}}\otimes \frac{\bf r}{\bf fr}\cong\frac{{\bf
r}^2}{(S-1){\bf r}+{\bf rfr}},
$$
and, since $F\cap (1+{\bf fr}+{\bf rf})=S$,
$$
Q/U\otimes Q\hookrightarrow \frac{\bf r}{{\bf fr}+{\bf rf}}\otimes
\frac{\bf r}{\bf fr}\cong\frac{{\bf r}^2}{{\bf fr}^2+{\bf rfr}}.
$$
The exact sequence (\ref{firstdialine}) implies that the left hand
quotients in (\ref{sph21}) and (\ref{sph22}) are naturally
isomorphic to $L_1{\sf SP}^2(R/S)$. Note that there is a natural
isomorphism
$$
L_1{\sf SP}^2(R/S)\cong L_1{\sf SP}^2(H_2(G)).$$
 To see this, observe that
$$
R/S\cong H_2(G)\oplus \mathcal F
$$with $\mathcal F$ a  free\ abelian\ group,
and the asserted statements (\ref{sph21}) and (\ref{sph22}) thus
follow from the cross-effect formula for the functor $L_1{\sf
SP^2}$ (see, for example, (10.3) \cite{DP:1961}):$$ L_1{\sf
SP}^2(A\oplus B)=L_1{\sf SP}^2(A)\oplus L_1{\sf SP}^2(B)\oplus
{\sf Tor}(A,\,B).$$

To prove (\ref{sph23}), observe that the sequence \ref{l3seq}
implies that there is an exact sequence
\begin{equation}\label{mono7} 0\to L_2{\mathfrak L}_s^3(R/S)\to
\frac{\gamma_3(R)}{\gamma_3(S)\gamma_4(R)}\to \frac{{\bf
r}^2}{{\bf r}^2{\bf f}+{\bf s}^2}\otimes \frac{\bf r}{\bf fr}.
\end{equation}
Thus the isomorphism (\ref{sph23}) follows from the natural
isomorphisms
$$
L_2{\mathfrak L}_s^3(R/S)=L_2{\mathfrak L}_s^3(H_2(G))
$$
and
$$
\frac{{\bf r}^2}{{\bf r}^2{\bf f}+{\bf s}^2}\otimes \frac{\bf
r}{\bf fr}=\frac{{\bf r}^3}{{\bf r}^2{\bf fr}+{\bf s}^2\bf r}.
$$
$\Box$
\par\vspace{.5cm}
Using an implication, on the torsion in $L_1{\sf SP}^2(H_2(G))$,
of the result of R. St$\ddot{o}$hr \cite{Stohr:1984} (see also Yu.
V. Kuzmin \cite{Kuzmin:1982}) on the torsion in
$F/[\gamma_c(R),\,F]$, we immediately have  the following result.
\par\vspace{.25cm}
\begin{corollary}
If  $R$ is a norml subgroup of a free group $F$, $c\geq 2$ an
integer, and  $S=\gamma_c(R)$, then
$$
\frac{F\cap (1+{\bf sfs}+([S,\,F]-1){\bf
s})}{\gamma_2([S,\,F])\gamma_3(S)}
$$
is a torsion group of  exponent dividing   $c^2$.
\end{corollary}

\par\vspace{.25cm}
\par\vspace{.25cm}
\begin{theorem}
If $R$ and $S$ are  normal subgroups of a free group $F$ with  $S$
normal in $ R$, then there is a natural isomorphism $$\frac{F\cap
(1+{\bf
rf}+{\bf fs}+{\bf f}^3)}{[F,S]R'\gamma_3(F)}\cong\\
{\sf Coker}\{{\sf Tor}((F/R)_{ab}, \,(F/S)_{ab})\to L_1{\sf
SP}^2((F/R)_{ab})\}.$$
\end{theorem}\par\vspace{.25cm}
\begin{proof}
Let us set
$$
E:=F/SF'=(F/S)_{ab},\quad I:=RF'/SF'.
$$
We then have natural isomorphisms
$$
\Lambda^2(E)\cong =\frac{\gamma_2(F)}{[F,\,S]\gamma_3(F)},\quad
\Lambda^2(E)/\underline{\Lambda^2(I)}\cong
\frac{\gamma_2(F)}{[F,\,S]R'\gamma_3(F)},
$$
and
$$
E/I\otimes E\cong\frac{\bf f}{{\bf r}+{\bf f}^2}\otimes \frac{\bf
f}{{\bf s}+{\bf f}^2}\cong\frac{{\bf f}^2}{{\bf rf}+{\bf fs}+{\bf
f}^3}.
$$
The assertion in the theorem thus  follows from Theorem
\ref{lambdalemma}.
\end{proof}\par\vspace{.25cm}

\begin{theorem} If $1\to R\to F\to G\to 1$ is a free presentation of a group $G$, and $n\geq 2$ an integer, then
 there is a natural isomorphism
 $$\frac{F\cap (1+{\bf f}((F'-1)\mathbb Z[F]\cap {\bf f}^n)+{\bf r}{\bf f}^n+{\bf f}^{n+2})}{[R,\,R,\,\underbrace{F,\,\ldots\,, F}_{n-1}](F''\cap
\gamma_{n+1}(F))\gamma_{n+2}(F)}\cong L_1{\sf SP}^{n+1}(G_{ab}).$$
\end{theorem}
\begin{proof}
By (\ref{sergey}), there is a natural isomorphism
\begin{multline*}
(F/R)_{ab}\otimes {\sf SP}^n(F_{ab})\cong\frac{\bf f}{{\bf r}+{\bf
f}^2}\otimes \frac{{\bf f}^n}{(F'-1)\mathbb Z[F]\cap {\bf
f}^n+{\bf f}^{n+1}}\cong\\ \frac{{\bf f}^{n+1}}{{\bf rf}^n+{\bf
f}((F'-1)\mathbb Z[F]\cap {\bf f}^n)+{\bf f}^{n+2}}
\end{multline*}
The asserted statement thus follows from the sequence (\ref{spn}).
\end{proof}

\par\vspace{.25cm}
 In particular, taking $n=2,\, 3$, the preceding theorem yields the following interesting result.
\begin{corollary} There are natural isomorphisms
\begin{align*}
& \frac{F\cap (1+{\bf f}(F'-1)+{\bf r}{\bf f}^2+{\bf
f}^4)}{[R,\,R,\,F]\gamma_4(F)}\cong L_1{\sf SP}^3(G_{ab}).\\
& \frac{F\cap (1+{\bf f}((F'-1)\mathbb Z[F]\cap {\bf f}^3)+{\bf
r}{\bf f}^3+{\bf f}^5)}{[R,\,R,\,F,\,F]F''\gamma_5(F)}\cong
L_1{\sf SP}^4(G_{ab}).
\end{align*}
\end{corollary}
\par\vspace{.25cm}
\begin{theorem} Let $R$ and $S$ be normal subgroups of a free group $F$. Then there is a natural isomorphism
\begin{equation}
 \frac{F\cap (1+{\bf f}^2{\bf r}^2+{\bf f}(R'-1)+{\bf s}{\bf r}^2)}{[R\cap (SF'),\,R\cap (SF'),\,R]\gamma_4(R)}\cong L_1{\sf
SP}^3\left(\frac{R}{R\cap (SF')}\right)
\end{equation}
\end{theorem}
\begin{proof}
The sequence (\ref{spn}) implies that the kernel of the natural
map
$$
\frac{\gamma_3(R)}{[R\cap (SF'),\,R\cap (SF'),\,R]\gamma_4(R)}\to
\frac{R}{R\cap (SF')}\otimes {\sf SP}^2(R_{ab})
$$
is $L_1{\sf SP}^3\left(\frac{R}{R\cap (SF')}\right)$. Thus  the
asserted statement follows from the natural embedding
$$
\frac{R}{R\cap (SF')}\otimes {\sf SP}^2(R_{ab})\hookrightarrow
\frac{\bf f}{{\bf s+f}^2}\otimes \frac{{\bf r}^2}{(R'-1)+{\bf
fr}^2}\cong \frac{{\bf fr}^2}{{\bf sr}^2+{\bf f}^2{\bf r}^2+{\bf
f}(R'-1)}.
$$
\end{proof}
\par\vspace{.25cm}

\begin{theorem} If $1\to R\to f\to G\to 1$ is a free presentation of a group $G$, then
\begin{equation}\label{ident1}
F\cap (1+{\bf f}^2{\bf r}^2+{\bf f}(R'-1))=[R\cap F',R\cap F',
R]\gamma_4(R),
\end{equation} and there is a natural isomorphism
\begin{equation}\label{ident2}
\frac{F\cap (1+{\bf rfr}^2+{\bf f}{\bf r}^3+{\bf
r}(R'-1))}{[[R,\,F],[R,\,F],R]\gamma_4(R)}\cong L_1{\sf
SP}^3(H_2(G)).
\end{equation}
\end{theorem}
\begin{proof}
Since $R/(R\cap F')$ is torsion-free, $L_1{\sf SP}^3(R/(R\cap
F'))=0$ and the sequence (\ref{spn}) implies that there is a
natural monomorphism \begin{equation}\label{mono4}
\frac{\gamma_3(R)}{[R\cap F', R\cap F',
R]\gamma_4(R)}\hookrightarrow R/(R\cap F')\otimes {\sf
SP}^2(R_{ab}).
\end{equation}
The identification (\ref{ident1}) follows from the natural
monomorphism
\begin{equation}
R/(R\cap F')\otimes {\sf SP}^2(R_{ab})\hookrightarrow \frac{\bf
f}{{\bf f}^2}\otimes \frac{{\bf r}^2}{(R'-1)+{\bf fr}^2}\cong
\frac{{\bf fr}^2}{{\bf f}^2{\bf r}^2+{\bf f}(R'-1)}.
\end{equation}
To prove (\ref{ident2}),  we first observe that there exists an
exact sequence
\begin{equation}\label{mono5}
0\to L_1{\sf SP}^3(R/[R,F])\to
\frac{\gamma_3(R)}{[[R,F],[R,F],R]\gamma_4(R)}\to R/[R,F]\otimes
{\sf SP}^2(R_{ab}).
\end{equation}
Therefore  the assertion  follows from the isomorphisms
$$
R/[R,F]\otimes {\sf SP}^2(R_{ab})\cong \frac{\bf r}{\bf
rf+fr}\otimes \frac{{\bf r}^2}{(R'-1)+{\bf fr}^2}\cong \frac{{\bf
r}^3}{{\bf r}(R'-1)+{\bf rfr}^2+{\bf fr}^3}
$$
and the fact that $L_1{\sf SP}^3(R/[R,\,F])=L_1{\sf
SP}^3(H_2(G))$.
\end{proof}
\section{Limits}
A theory of limits for functors on the category of free
presentation of groups is developed in \cite{IM:2014},
\cite{IM:2015}, \cite{MP:2015a}. For a group $G$, consider the
category $\mathcal E$ of free presentations
$$
1\to R\to F\to G\to 1.
$$
For any functor (also called a representation) $\mathcal F:
\mathcal E \mapsto \mathcal A$, its limit $\ilimit \mathcal F$
presents a well defined functor from the category of groups to
$\mathcal A$. To illustrate, let us recall one example from
\cite{MP:2015a}. As mentioned in the introduction, for a
polynomial functor of degree $n$, the intermediate derived
functors (from the first till the $(n-2)$nd) as a rule have a
complicated nature. In \cite{MP:2015a}, the authors obtained a
limit formula for such a functor, namely, $L_1{\sf SP}^3$:
\begin{equation}\label{sp3ident}
L_1{\sf SP}^3(G_{ab})= \ilimit
\frac{\gamma_2(F)}{[R',F]\gamma_3(F)}.
\end{equation}
In the present work, we make further contribution to the theory of
limits, with the help of the results obtained above.

There are two basic simple properties of limits which we will
use  (see \cite{IM:2014}):\\
1) The inverse limit is left exact; for an exact sequence of
representations $$\mathcal F\hookrightarrow \mathcal G\to \mathcal
H$$ there is a natural exact sequence of limits
$$
\ilimit \mathcal F\hookrightarrow \ilimit \mathcal G\to \ilimit
\mathcal H.
$$
2) For any representations $\mathcal F, \mathcal G$,
$$
\ilimit \mathcal F\otimes \mathcal G(F_{ab})=0;
$$
i.e, if a representation is the tensor product of some
representation with a functor which depends only on $F_{ab}$, its
limit is zero.
\begin{theorem}\label{main1}
\begin{align}
& \ilimit \frac{R'}{\gamma_2([R,F])\gamma_3(R)}\cong
L_1{\sf SP}^2(H_2(G)),\label{o1}\\
& \ilimit \frac{\gamma_3(R)}{[\gamma_2(R\cap
F'),R]\gamma_4(R)}=0,\label{o2}\\
& \ilimit \frac{\gamma_3(R)}{[\gamma_2([R,F]),
R]\gamma_4(R)}\cong L_1{\sf SP}^3(H_2(G)),\label{o3}\\
& \ilimit \frac{\gamma_3(R)}{\gamma_3([R,F])\gamma_4(R)}\cong
L_2{\mathfrak L}_s^3(H_2(G)),\label{o4}\\
& \ilimit \frac{\gamma_4(F)}{[R,\,R,\,F,\,F]F''\gamma_5(F)}\cong
L_1{\sf SP}^4(G_{ab}),\label{o5}\\
& \ilimit
\frac{\gamma_4(R)}{[[R,\,F],\,[R,\,F],\,R,\,R]R''\gamma_5(R)}\cong
L_1{\sf
SP}^4(H_2(G)),\label{o6}\\
& \ilimit \frac{R''}{\gamma_2([R,\,R,\,F])\gamma_3(R')}\cong
L_1{\sf SP}^2(H_2(G,{\sf SP}^2(\bf g)))\label{o7},
\end{align}
in particular, if $G$ is 2-torsion-free, then
$$
\ilimit \frac{R''}{\gamma_2([R,\,R,\,F])\gamma_3(R')}\cong L_1{\sf
SP}^2(H_4(G,\,\mathbb Z/2)).\label{o8}
$$
\end{theorem}
\begin{proof}
It is shown in the proof of Theorem \ref{main} that there is the
following exact sequence
$$
0\to L_1{\sf SP}^2(H_2(G))\to
\frac{R'}{\gamma_2([R,F])\gamma_3(R)}\to R/[R,F]\otimes R_{ab}.
$$
The Magnus embedding $R_{ab}\hookrightarrow \mathbb Z[G]\otimes
F_{ab}\twoheadrightarrow {\bf g}$ implies that
$$
\ilimit R/[R,F]\otimes R_{ab} \hookrightarrow \ilimit
R/[R,F]\otimes \mathbb Z[G]\otimes F_{ab}=0.
$$
The isomorphism (\ref{o1}) thus follows.

The monomorphism (\ref{mono4}) implies that there is a
monomorphism $$\frac{\gamma_3(R)}{[R\cap F', R\cap F',
R]\gamma_4(R)}\hookrightarrow F_{ab}\otimes {\sf SP}^2(R_{ab}),$$
and so the statement (\ref{o2}) follows.

The sequence (\ref{mono5}) implies that there is the following
exact sequence $$0\to L_1{\sf SP}^3(H_2(G))\to
\frac{\gamma_3(R)}{[[R,F],[R,F],R]\gamma_4(R)}\to R/[R,F]\otimes
{\sf SP}^2(F_{ab}),$$ and the isomorphism (\ref{o3}) follows.

The same arguments show that the sequence (\ref{mono7}) implies
(\ref{o4}). Further, the sequence (\ref{spn}) implies (\ref{o5}).
The sequence (\ref{spn}) implies that there is the following exact
sequence
$$
0\to L_1{\sf SP}^4(R/[R,F])\to
\frac{\gamma_4(R)}{[[R,F],[R,F],R,R]R''\gamma_5(R)}\to
R/[R,F]\otimes {\sf SP}^3(R_{ab}).
$$
Since
$$
\ilimit R/[R,F]\otimes {\sf SP}^3(R_{ab})\hookrightarrow \ilimit
R/[R,F]\otimes {\sf SP}^3(F_{ab})=0,
$$
the isomorphism (\ref{o6}) follows. The isomorphism (\ref{o7})
follows in the same way. To see it, we get first the isomorphism
$$
L_1{\sf SP}^2(R'/[R,R,F])\cong\ilimit
\frac{R''}{\gamma_2([R,\,R,\,F])\gamma_3(R')}.
$$
Now the identifications (\ref{o7}) and (\ref{o8}) follow from the
results of R. St\"ohr \cite{Stohr:1984}, which describe the
torsion of $R''/[R'',F]$.
\end{proof}

We end the paper with a problem. A detailed analysis analogous to
that done in \cite{MP:2015a} shows that, for any $n\geq 2,$ the
limit
$$
\ilimit \frac{\gamma_n(F)}{\gamma_n(R)\gamma_{n+1}(F)}
$$
can be identified with the $(n-1)$st derived functor of the $n$th
super-Lie power of $G_{ab}$. On the other hand,
$$
\ilimit
\frac{\gamma_n(F)}{[R,\underbrace{F,\,F,\,\ldots,\,F}_{n-1}\,]\gamma_{n+1}(F)}\cong\EuScript
L^n(G_{ab}).
$$
The identification (\ref{sp3ident}) suggests a conjecture that the
limit
\begin{equation}\label{unknown}
\ilimit
\frac{\gamma_n(F)}{[\underbrace{R,\ldots,\,R}_k,\underbrace{F,\,\ldots,\,F}_{n-k}\,]\gamma_{n+1}(F)}
\end{equation}
may be related to the $(k-1)$st derived functor of some
well-described polynomial functor of degree $n$ applied to
$G_{ab}$.

\vspace{.5cm}\noindent{\bf Problem.} Describe the functors
(\ref{unknown}) for all $n>3$ and $k=2,\dots, n-1$.


\begin{thebibliography}{7}

\bibitem{BM:2011}L. Breen and R. Mikhailov: Derived functors of non-additive functors and homotopy theory, {\it Algebr. Geom. Topol.}, {\bf 11} (2011), 327 - 415.
\bibitem{DP:1961} A. Dold and D. Puppe: Homologie nicht additiver Funktoren Anwendugen, {\it Annales de l'institut  Fourier}, tome {\bf 11}, n${^o}$. 6 (1961), 201 - 312.

\bibitem{hpin} S. Eilenberg and S. Mac Lane, \emph{On the groups $H(\Pi,n)$ II}, Ann. of Math. (2) 60 (1954) 49--139.

\bibitem{Gupta:87} Narain Gupta: {\it Free Group Rings}, Contemporary Mathematics, Vol. {\bf 66}, American Mathematical Society, 1987.

\bibitem{HMP:2009} M. Hartl, R. Mikhailov and I. B. S. Passi: Dimension quotients, J. Indian Math. Soc., New Ser. Spec. Centenary Vol., 63-107 (2007).

\bibitem{IM:2015}Sergei O. Ivanov and Roman Mikhailov: Higher limits, homology theories and fr-codes,   arXiv:1510.09044v1 [math.GR].


\bibitem{IM:2014} Sergei O. Ivanov and Roman Mikhailov: A higher limit approach to homology theories, {\it J. Pure and Appl. Algebra} {\bf 219} (2015), 1915--1939.

\bibitem{Kock:2001} B. \text{K$\ddot{o}$}ck: Computing the homology of Koszul complexes, {\it Trans. Amer. Math. Soc.}, {\bf 353} (2001), 3115 - 3147.

\bibitem{Kuzmin:1982} Yu. V. Kuzmin: On elements of finite order in free groups of some varieties. Mat. Sb. 119, 119-
131 (1982) [Russian].

\bibitem{MP:2009} Roman Mikhailov and Inder Bir S. Passi: {\it Lower Central and Dimension Series of Groups}, LNM Vol. {\bf 1952}, Springer 2009.

\bibitem{MP:2015a} Roman Mikhailov and Inder Bir S. Passi: Generalized dimension subgroups and derived functors, {\it J. Pure Appl. Algebra} {\bf 220} (2016),  2143--2163.



\bibitem{Stohr:1984} R. St$\ddot{o}$hr: On Gupta representations of central extensions, {\it Math. Z.} {\bf 187}, 259-267 (1984).

\bibitem{Vermani:1999} L. R. Vermani: On subgroups determined by ideals of an integral  group ring, Passi, I. B. S. (ed.),
{\it Algebra. Some recent advances}. Basel: Birkhauser. Trends in
Mathematics. 227-242 (1999).

\end{thebibliography}
\end{document}